\documentclass{amsart}
\usepackage{esint}
  \usepackage{paralist}
  \usepackage{graphics}
  \usepackage{epsfig}
  \usepackage[colorlinks=true]{hyperref}
\hypersetup{urlcolor=blue, citecolor=red}

\usepackage{calc}
\newtheorem{theorem}{Theorem}
\newtheorem{lemma}[theorem]{Lemma}
\newtheorem{proposition}{Proposition}
\theoremstyle{definition}
\newtheorem{definition}[theorem]{Definition}
\newtheorem{example}[theorem]{Example}
\usepackage{esint}

\title[Many particle approximation of the Aw-Rascle-Zhang model]
{Many particle approximation of the Aw-Rascle-Zhang second order model for vehicular traffic}

\author[M.\ Di Francesco, S.\ Fagioli and M.\ D.\ Rosini]{}

\subjclass{Primary: 35L65, 90B20.}
 \keywords{Aw-Rascle-Zhang model, second order models for vehicular traffics, many particle limit.}

 \email{marco.difrancesco@univaq.it}
 \email{simone.fagioli@univaq.it}
 \email{mrosini@umcs.lublin.pl}
\thanks{MDF is supported by the Italian MIUR-PRIN project $2012L5WXHJ\_003$. SF is partially supported by the Italian INdAM-GNAMPA 2015 mini-project: Analisi e stabilit\`a per modelli di equazioni alle derivate parziali nella matematica applicata. MDR is also partially supported by ICM Interdyscyplinarne Centrum Modelowania Matematycznego i Komputerowego, Uniwersytet Warszawski. The authors would like to thanks Giovanni Russo for comments and suggestions on the numerical part.}

\thanks{$^*$ Corresponding author: M.\ D.\ Rosini}

\allowdisplaybreaks

\begin{document}
\maketitle
\centerline{\scshape Marco Di Francesco and Simone Fagioli}
\medskip
{\footnotesize
\centerline{DISIM, Universit\`a degli Studi dell'Aquila,}
\centerline{via Vetoio 1 (Coppito), 67100 L’Aquila (AQ), Italy}
}

\medskip

\centerline{\scshape Massimiliano D.\ Rosini$^*$}
\medskip
{\footnotesize
\centerline{Instytut Matematyki, Uniwersytet Marii Curie-Sk\l odowskiej,}
\centerline{pl.\ Marii Curie-Sk\l odowskiej 1, 20-031 Lublin, Poland}
}

\begin{abstract}
We consider the follow-the-leader approximation of the Aw-Rascle-Zhang (ARZ) model for traffic flow in a multi population formulation. We prove rigorous convergence to weak solutions of the ARZ system in the many particle limit in presence of vacuum. The result is based on uniform ${\mathbf{BV}}$ estimates on the discrete particle velocity. We complement our result with numerical simulations of the particle method compared with some exact solutions to the Riemann problem of the ARZ system.
\end{abstract}

\section{Introduction}

The Aw, Rascle \cite{Aw2000916} and Zhang \cite{Zhang2002275} (ARZ) model is a second order system describing vehicular traffic.
In its continuum formulation, it can be written as the $2\times2$ system of conservation laws in one space dimension
\begin{equation}\label{eq:ARZ_intro}
\begin{cases}
\rho_t + \left(\rho \, v\right)_x = 0,&t> 0, ~ x\in {\mathbb{R}},
\\
\left[\rho \, (v+p(\rho))\right]_t + \left[\rho \, v \, (v+p(\rho))\right]_x = 0,&t> 0, ~ x\in {\mathbb{R}}.
\end{cases}
\end{equation}

The conserved variables $\rho$ and $[\rho \, (v+p(\rho))]$ describe respectively the \emph{density} and the \emph{generalized momentum} of the system. $v$ is the \emph{velocity}. The quantity
\[w \doteq v+p(\rho)\]
is called \emph{Lagrangian marker}. The function $p$ in \eqref{eq:ARZ_intro} is the \emph{pressure function} and accounts for drivers' reactions to the state of traffic in front of them.
While traffic flow is one of the main motivating applications behind the system \eqref{eq:ARZ_intro}, we see a growing interest nowadays on different contexts as crowd dynamics and bio-mathematics.

The instability near the vacuum state makes the mathematical theory for \eqref{eq:ARZ_intro} a challenging topic.
For this reason, as in \cite{BCM-order2,BCJMU-order2,godvik}, we study \eqref{eq:ARZ_intro} in the Riemann invariant coordinates $(v,w)$ and use the following expressions for the density
\begin{align*}
&\rho \doteq p^{-1}(w-v).
\end{align*}

It is well known since \cite{AKMR2002} and the earlier works \cite{gazis,prigogine} (see also a related result in \cite{moutari}) that the discrete Lagrangian counterpart of \eqref{eq:ARZ_intro} is provided by the second order \emph{follow-the-leader} system
\begin{equation}\label{eq:follow_the_leader_intro1}
  \begin{cases}
  \displaystyle{\dot{x}_i=V_i}\,, & \\
  \displaystyle{\dot{V}_i=p'\left(\frac{1}{x_{i+1}-x_i}\right) \frac{V_{i+1}-V_i}{(x_{i+1}-x_i)^2}}\,,&
  \end{cases}
\end{equation}
where $x_i(t)$ and $V_i(t)$ are location of the tail and speed of the $i$-th vehicle at time $t$. In terms of the discrete Lagrangian marker
\[w_i \doteq V_i+p\left(\frac{1}{x_{i+1}-x_i}\right),\]
the system \eqref{eq:follow_the_leader_intro1} reads in the simpler form
\begin{equation}\label{eq:follow_the_leader_intro2}
  \begin{cases}
  \displaystyle{\dot{x}_i=w_i-p\left(\frac{1}{x_{i+1}-x_i}\right)}, & \\
  \displaystyle{\dot{w}_i=0}\,.&
  \end{cases}
\end{equation}
The simpler form \eqref{eq:follow_the_leader_intro2} highlights the fact that the follow-the-leader system \eqref{eq:follow_the_leader_intro1} describes a particle system with \emph{many species}.
Hence \eqref{eq:follow_the_leader_intro1} is a microscopic \emph{multi population} model, in which the $i$-th vehicle has length $[1/p^{-1}(w_i)]$ and maximal speed $w_i$, which are constant in time and depend on the initial datum of the Lagrangian marker.

The goal of this paper is to approximate (under reasonable assumptions on $p$) the $2\times2$ system of nonlinear conservation laws \eqref{eq:ARZ_intro} with the first order microscopic model \eqref{eq:follow_the_leader_intro2}, that describes a multi population interacting many particle system. The technique we adopt is a slight variation of that introduced in \cite{MarcoMaxARMA2015}, which applies to the approximation of first order LWR models \cite{LWR1, LWR2} by a (simpler) first order version of the follow-the-leader system. Despite the second order nature of \eqref{eq:ARZ_intro}, the strategy developed in \cite{MarcoMaxARMA2015} applies also in this case (see also \cite{DF_fagioli_rosini_UMI, DF_fagioli_rosini_russo, DF_fagioli_rosini_russo_02} for other applications of the techniques introduced in \cite{MarcoMaxARMA2015}). 
This reveals that the multi-species nature of the ARZ model is quite relevant in the dynamics. Our rigorous results only deal with the convergence toward a weak solution to \eqref{eq:ARZ_intro}. The problem of the uniqueness of entropy solutions for \eqref{eq:ARZ_intro} is quite a hard task, and we do not address it here, see \cite{BCM-order2,BCJMU-order2}. Our main convergence result is stated in Theorem~\ref{thm:main} below.

We also perform numerical simulations that suggest that the solution of the microscopic model \eqref{eq:follow_the_leader_intro2} converges to a solution of the macroscopic model \eqref{eq:ARZ_intro} as the number of particles goes to infinity. In particular, we will make the tests considered in \cite{chalonsgoatin2007} to show that we do not have the spurious oscillations generated, for instance, by the Godunov method near contact discontinuities. We will also make the test considered in \cite{AKMR2002} to show that our algorithm is able to cope with the vacuum.

Our approach deeply differ from the one proposed in \cite{AKMR2002}.
Indeed, there the authors show how the ARZ model written in \emph{Lagrangian mass coordinates} can be viewed \emph{away from the vacuum} as the limit of a \emph{time discretization} of a \emph{second order} microscopic model as the number of vehicles increases, with a scaling in space and time (a zoom) for which the density and the velocity remain fixed.
On the contrary, we consider the ARZ model written in the Eulerian coordinates, so that the vacuum is a state eventually achieved by the solutions, and we introduce the underlying microscopic model without performing any time discretization.

We recall that the numerical transport-equilibrium scheme proposed in \cite{chalonsgoatin2007} is based on both a (Glimm) random sampling strategy and the Godunov method. As a consequence, this method is non-conservative. Moreover, it cannot be applied in the presence of the vacuum state.
On the contrary, our method is conservative and is able to cope with the vacuum. Let us finally underline that the approximate solutions constructed with both methods have sharp (without numerical diffusion) contact discontinuities and show numerical convergence. We remark that the model \eqref{eq:ARZ_intro} can be formulated also with a relaxation term with a prescribed equilibrium velocity, see \cite{AKMR2002}. We shall apply our approach to such an extended version of the ARZ model in a future work.

The present paper is structured as follows. In Section~\ref{sec:pre} we recall the basic properties of the discrete follow-the-leader model and of the continuum ARZ system. In particular we prove a discrete maximum principle in Lemma~\ref{lem:1} which was not present in the literature to our knowledge. In Section~\ref{sec:ato} we construct the atomization scheme and prove its convergence in the Theorem~\ref{thm:main}. In Section~\ref{sec:numerics} we perform numerical tests with simple Riemann problems, including cases with vacuum.

\section{Preliminaries}\label{sec:pre}

In this section we recall basic facts about the ARZ model \eqref{eq:ARZ_intro} and the discrete follow-the-leader system \eqref{eq:follow_the_leader_intro2}.

\subsection{The ARZ model}

Consider the Cauchy problem for the ARZ model \eqref{eq:ARZ_intro}
\begin{equation}\label{eq:CauchyARZ}
\begin{cases}
    \rho_t + \left(\rho \, v\right)_x=0,&t>0,~ x\in {\mathbb{R}},	
    \\
    (\rho\,w)_t + \left(\rho\,v\,w\right)_x=0,&t>0,~ x\in {\mathbb{R}},
    \\
    v(0,x)= \bar{v}(x),&x\in {\mathbb{R}},
    \\
    w(0,x)= \bar{w}(x),&x\in {\mathbb{R}},
\end{cases}
\end{equation}
where $(v,w)$ is the unknown variable and $(\bar{v},\bar{w})$ is the corresponding initial datum.
More precisely, $(v,w)$ belongs to $\mathcal{W} \doteq \left\{(v,w) \in {\mathbb{R}}_+^2 \colon v \le w\right\}$ and
\[\rho\doteq p^{-1}(w-v) \in {\mathbb{R}}_+\] is the corresponding density, where $p \in \mathbf{C^0}({\mathbb{R}}_+;{\mathbb{R}}_+)\cap\mathbf{C^2}(\left]0,+\infty\right[;{\mathbb{R}}_+)$ satisfies
\begin{align}\label{eq:p}
    &p(0^+) = 0,&
    &p'(\rho)>0&
    &\text{and}&
    &2\,p'(\rho)+\rho \, p''(\rho)>0&
    \text{for every }\rho>0.
\end{align}

\begin{example}[Examples of pressure functions]
In \cite{AKMR2002} the authors consider
\[
p(\rho) \doteq
\begin{cases}
\dfrac{v_{\rm ref}}{\gamma} \left[\dfrac{\rho}{\rho_m}\right]^\gamma,& \gamma>0\,,\\[7pt]
v_{\rm ref} \, \log\left[\dfrac{\rho}{\rho_m}\right],& \gamma=0\,,
\end{cases}
\]
where $\rho_m>0$ is the maximal density and $v_{\rm ref}>0$ is a reference velocity.
The above choice reduces to the original one proposed in \cite{Aw2000916} when $v_{\rm ref}/(\gamma\, \rho_m^\gamma)= 1$ and $\gamma>0$.
We also recall that in \cite{BerthelinDegondDelitalaRascle} the authors consider
\[
p(\rho) \doteq \left(\frac{1}{\rho} - \frac{1}{\rho_m}\right)^{-\gamma}.
\]
\end{example}

By definition, we have that the vacuum state $\rho=0$ corresponds to the half line $\mathcal{W}_0 \doteq \left\{(v,w)^T \in \mathcal{W} \colon  v=w \right\}$ and the non-vacuum states $\rho>0$ to $\mathcal{W}_0^c \doteq \mathcal{W}\setminus\mathcal{W}_0$.
As in \cite{BCM-order2,BCJMU-order2}, we consider initial data $(\bar{v},\bar{w}) \in {\mathbf{BV}}({\mathbb{R}};\mathcal{W})$ such that
\begin{align}\label{eq:physreas}
&((\bar{v},\bar{w})(x^-), (\bar{v},\bar{w})(x^+)) \in \mathfrak{G},&&x\in {\mathbb{R}},
\end{align}
where $(\bar{v},\bar{w})(x^-)$ and $(\bar{v},\bar{w})(x^+)$ denote the traces of $(\bar{v},\bar{w})$ along a Lipschitz curve of jump (see \cite{Volpert} for a precise formulation of the regularity of ${\mathbf{BV}}$ functions) and
\[
\mathfrak{G} \doteq \left\{
((v_\ell,w_\ell),(v_r,w_r)) \in \mathcal{W}^2
\colon\begin{array}{@{}l@{}}
\left.\begin{array}{l@{}}
(v_\ell,w_\ell) \in \mathcal{W}_0\\
(v_r,w_r) \in \mathcal{W}_0
\end{array}\right\}
\Rightarrow
(v_\ell,w_\ell) = (v_r,w_r)
\\
\left.\begin{array}{l@{}}
\text{and}
\end{array}\right.
\\
\left.\begin{array}{l@{}}
(v_\ell,w_\ell) \in \mathcal{W}_0^c\\
(v_r,w_r) \in \mathcal{W}_0
\end{array}\right\}
\Rightarrow
(v_r,w_r) = \left(w_\ell, w_\ell\right)
\end{array}
\right\}.
\]
The introduction of the condition \eqref{eq:physreas} to select the physically reasonable initial data of \eqref{eq:CauchyARZ} in the Riemann invariant coordinates is motivated in \cite[Remark~2.1]{BCM-order2}. We emphasise that such condition is not needed in the proof of our main analytical result. However, \eqref{eq:physreas} is partly motivated by our numerical tests.

Following \cite{BCM-order2}, we use the following definition for weak solutions of \eqref{eq:ARZ_intro}.
\begin{definition}[Weak solutions]\label{def:weak}
Let $(\bar{v},\bar{w})\in \mathbf{L^\infty}({\mathbb{R}};\,\mathcal{W})$. We say that a function $(v,w)\in \mathbf{L^\infty}({\mathbb{R}}_+\times {\mathbb{R}};\,\mathcal{W})\cap \mathbf{C^0}({\mathbb{R}}_+;\,\mathbf{L_{loc}^1}({\mathbb{R}};\,\mathcal{W}))$ is a weak solution of \eqref{eq:CauchyARZ} if it satisfies the initial condition $(v(0,x),w(0,x))=(\bar{v}(x),\bar{w}(x))$ for a.e.\ $x\in {\mathbb{R}}$ and for any test function $\phi\in \mathbf{C_c^\infty}(]0,+\infty[\times {\mathbb{R}};\,{\mathbb{R}})$
\begin{equation}\label{eq:weak_sol}
  \int_{{\mathbb{R}}_+}\int_{\mathbb{R}} p^{-1}(v,w) \, (\phi_t + v \, \phi_x) \begin{pmatrix} 1 \\ w\end{pmatrix} \,{\rm{d}} x \,{\rm{d}} t = \begin{pmatrix} 0 \\ 0\end{pmatrix}.
\end{equation}
\end{definition}

For the existence of weak solutions to \eqref{eq:CauchyARZ} away from the vacuum we refer to \cite{ferreira}, see \cite{godvik} for the existence with vacuum, and \cite{bagnerini} for the existence of entropy weak solutions.

Let us briefly recall the main properties of the solutions to \eqref{eq:CauchyARZ}. If the initial density $\bar{\rho} \doteq p^{-1}(\bar{w}-\bar{v})$ has compact support, then the support of $\rho$ has finite speed of propagation. The maximum principle holds true in the Riemann invariant coordinates $(v,w)$, but not in the conserved variables $(\rho,\rho\,w)$ as a consequence of hysteresis processes. Moreover, the total space occupied by the vehicles (if they were packed ``nose to tail'') is time independent: $\int_{\mathbb{R}}  \rho(t,x) \,{\rm{d}}  x = M \doteq \int_{\mathbb{R}}  \bar \rho(x) \,{\rm{d}}  x$ for all $t\ge0$.

We remark that a simple byproduct of our result is an alternative proof of the existence of weak solutions in the sense of Definition~\ref{def:weak}.

\subsection{Follow-the-leader model}

Multi population microscopic models of vehicular traffic are typically based on the so called Follow-The-Leader (FTL) model.

Consider $[N+1]$ ordered particles localised on ${\mathbb{R}}$. Denote by $t \mapsto x_i(t)$ the position of the $i$-th particle for $i = 0, \ldots, N$. Then, according to the FTL model, the evolution of the particles (which mimics the evolution of the position of $[N+1]$ vehicles along the road) is described inductively by the following Cauchy problem for a system of ordinary differential equations
\begin{equation}\label{eq:FTL_intro}
\begin{cases}
  \dot{x}_N(t) = w_{N-1} ,\\
  \dot{x}_i(t) = v_i\left(\dfrac{1}{x_{i+1}(t)-x_i(t)}\right),&
  i=0,\ldots,N-1,\\
  x_i(0) = \bar{x}_i ,&
  i=0,\ldots,N,
\end{cases}
\end{equation}
where
\[v_i(\rho)\doteq w_i - p(\rho),\] $w_0,\ldots,w_N$ are $[N+1]$ strictly positive constants, $p \in \mathbf{C^2}({\mathbb{R}}_+;{\mathbb{R}}_+)$ satisfies \eqref{eq:p}, and $\bar{x}_0 <  \ldots < \bar{x}_N$ are the initial positions of the particles.
The quantity $w_i = v_i(0)$ is the maximum possible velocity allowed for the $i$-th particle.
Clearly, only the leading particle $x_N$ reaches its maximal velocity $w_N$, as the vacuum state is achieved only ahead of $x_N$.

The quantity $R_i \doteq p^{-1}(w_i)>0$ is the \emph{maximum discrete density} of the $i$-th particle, so that $R_i^{-1}$ is the \emph{length} of the $i$-th particle, $v_i(R_i)=0$ and we assume that at time $t=0$
\begin{align}\label{eq:FTL_assumption1}
    &\bar{x}_{i+1} - \bar{x}_{i} \ge \dfrac{1}{R_i},& i = 0, \ldots, N-1.
\end{align}

System \eqref{eq:FTL_intro} can be solved inductively starting from $i=N$. Indeed, we immediately deduce that
\[x_{N}(t) = \bar{x}_{N} + w_{N-1} \,t .\]
Then, we can compute $t \mapsto x_i(t)$ once we know $t \mapsto x_{i+1}(t)$. In fact, according with the system \eqref{eq:FTL_intro} the velocity of the $i$-th particle depends on its distance from the $(i+1)$-th particle alone via the smooth velocity map $v_i$.
In order to ensure that the (unique) solution to \eqref{eq:FTL_intro} exists globally in time, we need to prove that the distances $[x_{i+1}(t)-x_i(t)]$ never degenerate.
This is proven in the next lemma, which extends a similar one in \cite{MarcoMaxARMA2015} to our multi population system.

\begin{lemma}[Discrete maximum principle]\label{lem:1}
    For all $i = 0, \ldots, N-1$, we have
    \begin{align*}
        &\dfrac{1}{R_i} \le x_{i+1}(t)-x_i(t) \le \bar{x}_N - \bar{x}_{0} + w_{N-1} \,t &\text{for all times }t\ge0.
    \end{align*}
\end{lemma}
\begin{proof}
We first prove the lower bound.
At time $t=0$ the lower bound is satisfied because of \eqref{eq:FTL_assumption1}. We shall prove that
\begin{align}\label{eq:ast}\tag{$\ast$}
&\inf_{t\geq0} \left[\vphantom{\vec{V}} x_{i+1}(t)-x_i(t)\right] \geq \dfrac{1}{R_i},&
&i=0,\ldots,N-1,
\end{align}
by a recursive argument on $i$. The statement is true for $i=N-1$. Indeed, since $\dot{x}_{N-1}(t) \leq w_{N-1} = \dot{x}_{N}(t)$ we have that $x_{N-1}(t)<x_N(t)$ for all $t\geq0$ and
\begin{align*}
   x_N(t)-x_{N-1}(t)
   =&~ \bar{x}_N-\bar{x}_{N-1} + \int_0^t \left[w_{N-1}-v_{N-1}\left(\dfrac{1}{x_N(s)-x_{N-1}(s)}\right)\right] {\rm{d}}  s\\
   =&~ \bar{x}_N-\bar{x}_{N-1} + \int_0^t p\left(\dfrac{1}{x_N(s)-x_{N-1}(s)}\right) {\rm{d}}  s
   \geq \bar{x}_N-\bar{x}_{N-1}
   \geq \dfrac{1}{R_{N-1}},
\end{align*}
because $p(\rho)>0$ for all $\rho>0$. Assume now that
\[
   \inf_{t\geq0} \left[\vphantom{\vec{V}} x_{i+2}(t)-x_{i+1}(t)\right] \geq \dfrac{1}{R_{i+1}}
\]
and by contradiction that there exist $t_2>t_1\ge0$ such that
\begin{align}
\nonumber
&x_{i+1}(t)-x_i(t) > \dfrac{1}{R_i}&
&\text{for all }t \in \left[0,t_1\right[,
\\
\nonumber
&x_{i+1}(t_1)-x_i(t_1) = \dfrac{1}{R_i},
\\ \label{eq:star}\tag{$\star$}
&0<x_{i+1}(t)-x_i(t) < \dfrac{1}{R_i}&
&\text{for all }t \in \left]t_1, t_2\right].
\end{align}
Since $v_i$ is strictly decreasing, for any $t \in \left]t_1, t_2\right]$ we have that $\dot{x}_{i}(t) < v_{i}(R_{i}) =  0 = v_{i+1}(R_{i+1}) \leq \dot{x}_{i+1}(t)$.
Consequently, for any $t \in \left]t_1, t_2\right]$
\begin{align*}
&x_{i}(t)
\leq
x_{i}(t_1),&
&x_{i+1}(t)
\geq
x_{i+1}(t_1),
\end{align*}
and therefore
\[
x_{i+1}(t) - x_i(t)
\geq
x_{i+1}(t_1) - x_i(t_1) = \dfrac{1}{R_{i}},
\]
which contradicts \eqref{eq:star}.
Hence, \eqref{eq:ast} is satisfied and the lower bound is proven.

Finally, the upper bound easily follows from the lower bound.
Indeed, by the first equation of \eqref{eq:FTL_intro} and the lower bound $1 \le [x_1(t)-x_0(t)] \, R_0$ just proved, we have that $x_{N}(t) = \bar{x}_N + w_{N-1} \,t$ and $x_0(t)\ge \bar{x}_0 + v_0(R_0) \, t = \bar{x}_0$.
\end{proof}
We emphasise that the above discrete maximum principle is a direct consequence of the \emph{transport} nature behind the FTL system \eqref{eq:FTL_intro}, similarly to what happens in the first order FTL system considered in \cite{MarcoMaxARMA2015}. Indeed, the global bound for the discrete density is propagated from the last particle $x_N$ back to all the other particles, as emphasised by the recursive argument in the proof of Lemma~\ref{lem:1}.

\section{The atomization scheme and the strong convergence}\label{sec:ato}
We now introduce our atomization scheme for the Cauchy problem \eqref{eq:CauchyARZ}.
Let $(\bar{v},\bar{w}) \in {\mathbf{BV}}({\mathbb{R}};\mathcal{W})$ satisfy \eqref{eq:physreas} and such that $\bar\rho \doteq p^{-1}(\bar{w}-\bar{v})$ belongs to $\mathbf{L^1}({\mathbb{R}};{\mathbb{R}}_+)$ and is compactly supported. Denote by $\bar{x}_{\min} < \bar{x}_{\max}$ the extremal points of the convex hull of the compact support of $\bar\rho$, namely \[\bigcap_{\left[a,b\right]\supseteq{\mathrm{supp}}\left(\bar\rho\right)} \left[a,b\right] = \left[\bar{x}_{\min}, \bar{x}_{\max}\right].\]
Fix $n \in {\mathbb{N}}$ sufficiently large. Let $M \doteq \left\|\bar\rho\right\|_{\mathbf{L^1}({\mathbb{R}})} > 0$ and split the subgraph of $\bar \rho$ in $N_n \doteq 2^n$ regions of measure $\kappa_n \doteq 2^{-n}M$ as follows. Set
\begin{subequations}\label{eq:initial_ftl}
\begin{equation}
  \bar{x}_0^n \doteq \bar{x}_{\min},
\end{equation}
and recursively
\begin{align}
  &\bar{x}_i^n \doteq \sup\left\{x\in {\mathbb{R}} \, \colon\,\int_{\bar{x}^n_{i-1}}^x\bar\rho(x) \,{\rm{d}} x<\kappa_n\right\},&
  i=1,\ldots,N_n.
\end{align}
\end{subequations}
It is easily seen that $\bar{x}^n_{N_n}=\bar{x}_{\max}$, and $\bar{x}_{N_n-i}^n = \bar{x}_{N_{n+m}-2^mi}^{n+m}$ for all $i=0,\ldots,N_n$.
We approximate then the initial Lagrangian marker $\bar{w}$ by taking
\begin{align}\label{eq:dw}
&\bar{w}_i^n \doteq \underset{[\bar{x}_i^n,\bar{x}_{i+1}^n]}{\rm ess\,sup} (\bar{w}),
&i=0,\ldots,N_n-1.
\end{align}
We have then that the assumption \eqref{eq:FTL_assumption1} is satisfied as follows,
\begin{align*}
  &\kappa_n=\int_{\bar{x}_i^n}^{\bar{x}_{i+1}^n} \bar\rho(x) \,{\rm{d}} x \le \left(\bar{x}_{i+1}^n - \bar{x}_{i}^n\right) R_i^n ,&
  i=0,\ldots,N_n-1,
\end{align*}
with $R_i^n \doteq p^{-1}(\bar{w}_i^n)$.
We take the values $\bar{x}_0^n,\ldots,\bar{x}_{N_n}^n$ as the initial positions of the $[N_n+1]$ particles in the $n$--depending FTL model
\begin{equation}\label{eq:ftl}
\begin{cases}
    \dot x_{N_n}^n(t) = \bar{w}_{N_n-1}^n ,\\[7pt]
    \dot x_i^n(t) = v_i^n \left(\dfrac{\kappa_n}{x_{i+1}^n(t) - x_i^n(t)}\right) , & i=0, \ldots, N_n-1 ,\\[7pt]
    x_i^n(0) = \bar{x}_i^n , & i=0, \ldots, N_n ,
\end{cases}
\end{equation}
where
\begin{align}\label{eq:vni}
&v_i^n(\rho) \doteq \bar{w}_i^n - p(\rho),&
    i=0,\ldots,N_n-1.
\end{align}
The existence of a global-in-time solution to \eqref{eq:ftl} follows from Lemma~\ref{lem:1} with $v_i(\rho)$ replaced by $v^n_i(\kappa_n\,\rho)$. Moreover, from \eqref{eq:ftl} we immediately deduce that
\begin{align*}
    x_{N_n}^n(t) = \bar{x}_{\max} + \bar{w}_{N_n-1}^n \,t .
\end{align*}
Finally, since $v_i^n$ is decreasing, and its argument $\kappa_n/[x_{i+1}^n(t)-x_i^n(t)]$ is always bounded above by $R_i^n$, we have
\[x_0^n(t)\geq \bar{x}_{\min} + v_0(R_0^n) \, t = \bar{x}_{\min}.\]
By introducing in \eqref{eq:ftl} the new variable
\begin{align}\label{eq:Deltan}
    &y^n_i(t) \doteq \dfrac{\kappa_n}{x_{i+1}^n(t) - x_i^n(t)},&
    i=0,\ldots,N_n-1,
\end{align}
we obtain
\begin{equation}\label{eq:dyi}
\begin{cases}
    \dot{y}^n_{N_n-1} = -\dfrac{(y^n_{N_n-1})^2}{\kappa_n} \, p(y^n_{N_n-1}),
    \\
    \dot{y}^n_i = -\dfrac{(y^n_i)^2}{\kappa_n} \left[v_{i+1}^n(y^n_{i+1})-v_i^n(y^n_i)\right], & i=0, \ldots, N_n-2 ,
    \\
    y^n_i(0) = \bar{y}^n_i \doteq \dfrac{\kappa_n}{\bar{x}_{i+1}^n - \bar{x}_i^n} ,&
    i=0,\ldots,N_n-1.
\end{cases}
\end{equation}
Observe that $\kappa_n/\left[\bar{x}_{\max} - \bar{x}_{\min} + \bar{w}_{N_n-1}^n \, t\right] \le y^n_i(t) \le R_i^n$ for all $t\ge0$ in view of Lemma~\ref{lem:1}. The quantity $y^n_i$ can be seen as a discrete version of the density $\rho$ in Lagrangian coordinates, and \eqref{eq:dyi} is the discrete Lagrangian version of the Cauchy problem \eqref{eq:CauchyARZ}.

Define the piecewise constant (with respect to $x$) Lagrangian marker
\begin{align}\label{eq:hwn}
&W^n(t,x) \doteq
\begin{cases}
\bar{w}^n_0 &\text{if } x \in \left]-\infty, x_0^n(t)\right[,
\\
\bar{w}^n_i&\text{if } x \in \left[x_i^n(t), x_{i+1}^n(t)\right[,~ i = 0, \ldots, N_n-1,
\\
\bar{w}^n_{N_n-1}&\text{if } x \in \left[x^n_{N_n}(t),+\infty\right[,
\end{cases}
\intertext{and the piecewise constant (with respect to $x$) velocity}
\label{eq:hvn}
&V^n(t,x) \doteq
\begin{cases}
\bar{w}^n_0&\text{if } x \in \left]-\infty, x_0^n(t)\right[,
\\
v^n_i(y^n_i(t))&\text{if } x \in \left[x_i^n(t), x_{i+1}^n(t)\right[,~ i = 0, \ldots, N_n-1,
\\
\bar{w}^n_{N_n-1}&\text{if } x \in \left[x^n_{N_n}(t),+\infty\right[.
\end{cases}
\end{align}

\begin{proposition}[Definition of $w$]\label{prop:w}
Let $W^n$ be defined as in \eqref{eq:hwn}.
Then there exists a function $w \in \mathbf{L_{loc}^\infty}\left({\mathbb{R}}_+ \times {\mathbb{R}}\right)$ such that
\begin{align*}
&(W^n)_{n\in{\mathbb{N}}} \text{ converges to } w
\text{ in } \mathbf{L_{loc}^1}\left({\mathbb{R}}_+ \times {\mathbb{R}}\right),
\end{align*}
and for any $t,s\geq0$
\begin{subequations}\label{eq:westimates}
\begin{align}
&{\mathrm{TV}}\left[w(t)\right] \le {\mathrm{TV}}\left[\bar{w}\right] ,\\
&\left\|w(t)\right\|_{\mathbf{L^\infty}({\mathbb{R}})} \le \left\|\bar{w}\right\|_{\mathbf{L^\infty}({\mathbb{R}})} ,\\
&\int_{{\mathbb{R}}} \left|w(t,x) - w(s,x)\right| \,{\rm{d}} x \le {\mathrm{TV}}\left[\bar{w}\right] \, \left\|\bar{w}\right\|_{\mathbf{L^\infty}({\mathbb{R}})} \, \left|t-s\right| .
\end{align}
\end{subequations}
\end{proposition}
\begin{proof}
Directly from the definition \eqref{eq:hwn} of $W^n$ follows that for any $t\geq0$
\begin{align*}
&{\mathrm{TV}}\left[W^n(t)\right] \le {\mathrm{TV}}\left[\bar{w}\right] &
\text{and}&
&\left\|W^n(t)\right\|_{\mathbf{L^\infty}({\mathbb{R}})} \le \left\|\bar{w}\right\|_{\mathbf{L^\infty}({\mathbb{R}})} .
\end{align*}
Moreover, since the speed of propagation of the particles is bounded by $\left\|\bar{w}\right\|_{\mathbf{L^\infty}({\mathbb{R}})}$, by the above uniform bound for the total variation we have that
\[
\int_{{\mathbb{R}}} \left|W^n(t,x) - W^n(s,x)\right| \,{\rm{d}} x
\le
{\mathrm{TV}}\left[\bar{w}\right] \, \left\|\bar{w}\right\|_{\mathbf{L^\infty}({\mathbb{R}})} \, \left|t-s\right|.
\]
Hence, by applying Helly's theorem in the form \cite[Theorem~2.4]{BressanBook}, up to a subsequence, $(W^n)_{n\in{\mathbb{N}}}$ converges in $\mathbf{L_{loc}^1}\left({\mathbb{R}}_+ \times {\mathbb{R}}\right)$ to a function $w$ which is right continuous with respect to $x$ and satisfies \eqref{eq:westimates}.

Finally, observe that by the definition in \eqref{eq:dw} of $\bar{w}_i^n$, for any $n\in{\mathbb{N}}$ we have
\begin{align*}
&\underset{[\bar{x}_{\min},\bar{x}_{\max}]}{\rm ess\,inf} (\bar{w}) \le W^{n+1}(t,x) \le W^n(t,x) \le \underset{[\bar{x}_{\min},\bar{x}_{\max}]}{\rm ess\,sup} (\bar{w})
&\text{for all }(t,x)\in {\mathbb{R}}_+\times {\mathbb{R}} .
\end{align*}
Therefore the whole sequence $(W^n)_{n\in{\mathbb{N}}}$ converges to $w$ and a.e.\ on ${\mathbb{R}}_+\times {\mathbb{R}}$.
\end{proof}

\begin{proposition}[Definition of $v$]\label{prop:v}
Let $V^n$ be defined as in \eqref{eq:hvn}.
Then there exists a function $v \in \mathbf{L_{loc}^\infty}\left({\mathbb{R}}_+ \times {\mathbb{R}}\right)$, such that
\begin{align*}
&(V^n)_{n\in{\mathbb{N}}} \text{ converges up to a subsequence to } v
\text{ in } \mathbf{L_{loc}^1}({\mathbb{R}}_+ \times {\mathbb{R}}),
\end{align*}
and for any $t,s\geq0$
\begin{subequations}\label{eq:vestimates}
\begin{align}
&{\mathrm{TV}}\left[v(t)\right] \le C_v \doteq 2\,\left\|\bar{w}\right\|_{\mathbf{L^\infty}({\mathbb{R}})} + {\mathrm{TV}}[\bar{w}] + {\mathrm{Lip}}(p) \, {\mathrm{TV}}[\bar{\rho}],\\
&\left\|v(t)\right\|_{\mathbf{L^\infty}({\mathbb{R}})} \le \left\|\bar{w}\right\|_{\mathbf{L^\infty}({\mathbb{R}})} ,\\
&\int_{{\mathbb{R}}} \left|v(t,x) - v(s,x)\right| {\rm{d}} x \le C_v \, \left\|\bar{w}\right\|_{\mathbf{L^\infty}({\mathbb{R}})} \, \left|t-s\right| .
\end{align}
\end{subequations}
\end{proposition}

\begin{proof}
For notational simplicity, we shall omit the dependence on $t$ whenever not necessary.
By construction, see \eqref{eq:FTL_assumption1}, \eqref{eq:Deltan} and \eqref{eq:hvn}, we have that
\begin{align*}
{\mathrm{TV}}\left[V^n(0)\right]
=&~
\left|\bar{w}_0^n-v_0^n\left(\bar{y}_0^n\right)\right|
+\sum_{i=0}^{N_n-2} \left|v_i^n(\bar{y}_i^n)-v_{i+1}^n(\bar{y}_{i+1}^n)\right|
+
\left|v_{N_n-1}^n\left(\bar{y}_{N_n-1}^n\right)-\bar{w}_{N_n-1}^n\right|
\\\leq&~
p\left(\bar{y}_0^n\right)
+\sum_{i=0}^{N_n-2} \left|\bar{w}_i^n-\bar{w}_{i+1}^n\right|
+{\mathrm{Lip}}(p) \sum_{i=0}^{N_n-2} \left|\bar{y}_i^n-\bar{y}_{i+1}^n\right|
+
p\left(\bar{y}_{N_n-1}^n\right)
\\\leq&~
p\left(R_0^n\right)
+{\mathrm{TV}}[\bar{w}]
+{\mathrm{Lip}}(p) \sum_{i=0}^{N_n-2} \left|\fint_{\bar{x}_{i}^n}^{\bar{x}_{i+1}^n} \bar\rho(y) {\,{\rm{d}}}y - \fint^{\bar{x}_{i+2}^n}_{\bar{x}_{i+1}^n} \bar\rho(y) {\,{\rm{d}}}y\right|
+
p\left(R_{N-1}^n\right)
\\\leq&~
2\,\left\|\bar{w}\right\|_{\mathbf{L^\infty}({\mathbb{R}})}
+{\mathrm{TV}}[\bar{w}]
+{\mathrm{Lip}}(p) \, {\mathrm{TV}}[\bar{\rho}].
\end{align*}
Moreover
\begin{align*}
&~\dfrac{{\rm{d}}~}{{\rm{d}} t} {\mathrm{TV}}\left[V^n(t)\right]
=
\dfrac{{\rm{d}}~}{{\rm{d}} t} \left[
\left|\bar{w}_0^n-v_0^n(y_0^n)\right|
+\sum_{i=0}^{N_n-2} \left|v_i^n(y_i^n)-v_{i+1}^n(y_{i+1}^n)\right|
+\left|v_{N_n-1}^n(y_{N_n-1}^n)-\bar{w}_{N_n-1}^n\right|
\right]
\\
=&~
\dfrac{{\rm{d}}~}{{\rm{d}} t} \left[
p(y_0^n)
+\sum_{i=0}^{N_n-2} \left|v_i^n(y_i^n)-v_{i+1}^n(y_{i+1}^n)\right|
+p(y_{N_n-1}^n)
\right]
\\
=&~
p'(y_0^n) \, \dot{y}_0^n
+\sum_{i=0}^{N_n-2} {\textrm{sgn}}\left[v_i^n(y_i^n)-v_{i+1}^n(y_{i+1}^n)\right] \left[p'(y_{i+1}^n) \, \dot{y}_{i+1}^n - p'(y_i^n) \, \dot{y}_i^n\right]
+p'(y_{N_n-1}^n) \, \dot{y}_{N_n-1}^n
\\
=&~
p'(y_0^n) \, \dot{y}_0^n
+\sum_{i=1}^{N_n-1} {\textrm{sgn}}\left[v_{i-1}^n(y_{i-1}^n)-v_i^n(y_i^n)\right] p'(y_i^n) \, \dot{y}_i^n
\\
&~-\sum_{i=0}^{N_n-2} {\textrm{sgn}}\left[v_i^n(y_i^n)-v_{i+1}^n(y_{i+1}^n)\right] p'(y_i^n) \, \dot{y}_i^n
+p'(y_{N_n-1}^n) \, \dot{y}_{N_n-1}^n
\\
=&~
\left[\vphantom{\vec{V}} 1-{\textrm{sgn}}\left[v_0^n(y_0^n)-v_1^n(y_1^n)\right]\right] p'(y_0^n) \, \dot{y}_0^n
+\left[\vphantom{\vec{V}} 1+{\textrm{sgn}}\left[v_{N_n-2}^n(y_{N_n-2}^n)-v_{N_n-1}^n(y_{N_n-1}^n)\right]\right] p'(y_{N_n-1}^n) \, \dot{y}_{N_n-1}^n
\\&~
+\sum_{i=1}^{N_n-2} \left[\vphantom{\vec{V}} {\textrm{sgn}}\left[v_{i-1}^n(y_{i-1}^n)-v_i^n(y_i^n)\right] - {\textrm{sgn}}\left[v_i^n(y_i^n)-v_{i+1}^n(y_{i+1}^n)\right]\right] p'(y_i^n) \, \dot{y}_i^n.
\end{align*}
We claim that the latter right hand side above is not positive.
Indeed \eqref{eq:dyi} implies that the following quantities are not positive
\begin{align*}
&\left[\vphantom{\vec{V}} 1-{\textrm{sgn}}\left[v_0^n(y_0^n)-v_1^n(y_1^n)\right]\right] p'(y_0^n) \, \dot{y}_0^n
=
-\left[\vphantom{\vec{V}} 1-{\textrm{sgn}}\left[v_0^n(y_0^n)-v_1^n(y_1^n)\right]\right] p'(y_0^n) \, \dfrac{(y_0^n)^2}{\kappa_n} \left[v_1^n(y_1^n)-v_0^n(y_0^n)\right],
\\[7pt]
&\left[\vphantom{\vec{V}} 1+{\textrm{sgn}}\left[v_{N_n-2}^n(y_{N_n-2}^n)-v_{N_n-1}^n(y_{N_n-1}^n)\right]\right] p'(y_{N_n-1}^n) \, \dot{y}_{N_n-1}^n
=\\
=&-\left[\vphantom{\vec{V}} 1+{\textrm{sgn}}\left[v_{N_n-2}^n(y_{N_n-2}^n)-v_{N_n-1}^n(y_{N_n-1}^n)\right]\right] p'(y_{N_n-1}^n) \, \dfrac{(y_{N_n-1}^n)^2}{\kappa} \, p(y_{N_n-1}^n),
\\[7pt]
&\left[\vphantom{\vec{V}} {\textrm{sgn}}\left[v_{i-1}^n(y_{i-1}^n)-v_i^n(y_i^n)\right] - {\textrm{sgn}}\left[v_i^n(y_i^n)-v_{i+1}^n(y_{i+1}^n)\right]\right] p'(y_i^n) \, \dot{y}_i^n
=\\
=&-\left[\vphantom{\vec{V}} {\textrm{sgn}}\left[v_{i-1}^n(y_{i-1}^n)-v_i^n(y_i^n)\right] - {\textrm{sgn}}\left[v_i^n(y_i^n)-v_{i+1}^n(y_{i+1}^n)\right]\right]
p'(y_i^n) \, \dfrac{(y_i^n)^2}{\kappa_n} \left[v_{i+1}^n(y_{i+1}^n)-v_i^n(y_i^n)\right].
\end{align*}
Therefore, ${\mathrm{TV}}\left[V^n(t)\right] \leq {\mathrm{TV}}\left[V^n(0)\right] \leq C_v \doteq 2\,\left\|\bar{w}\right\|_{\mathbf{L^\infty}({\mathbb{R}})} + {\mathrm{TV}}[\bar{w}] + {\mathrm{Lip}}(p) \, {\mathrm{TV}}[\bar{\rho}]$ for all $t\geq0$.
The estimate $\left\|V^n(t)\right\|_{\mathbf{L^\infty}({\mathbb{R}})} \le \left\|\bar{w}\right\|_{\mathbf{L^\infty}({\mathbb{R}})}$ is obvious.
Moreover, since the speed of propagation of the particles is bounded by $\left\|\bar{w}\right\|_{\mathbf{L^\infty}({\mathbb{R}})}$, by the above uniform bound for the total variation we have that
\[
\int_{{\mathbb{R}}} \left|V^n(t,x) - V^n(s,x)\right| {\rm{d}} x
\le
C_v \, \left\|\bar{w}\right\|_{\mathbf{L^\infty}({\mathbb{R}})} \, \left|t-s\right|.
\]
Hence, by applying Helly's theorem in the form \cite[Theorem~2.4]{BressanBook}, up to a subsequence, $(V^n)_{n\in{\mathbb{N}}}$ converges in $\mathbf{L_{loc}^1}\left({\mathbb{R}}_+ \times {\mathbb{R}}\right)$ to a function $v$ which is right continuous with respect to $x$ and satisfies \eqref{eq:vestimates}.
\end{proof}

We are now ready to prove our main result. Let us define some technical machinery. We introduce the piecewise constant density
\[\rho^n(t,x) \doteq p^{-1}(W^n(t,x)-V^n(t,x)) = \sum_{i=1}^{N_n-1} y_i^n(t) \, \chi_{[x_i^n(t),x_{i+1}^n(t)[}(x).\]
We set
\[\mathcal{M}_M \doteq \left\{\vphantom{I^{\int}} \mu \hbox{ Radon measure on ${\mathbb{R}}$ with compact support } \colon \mu\geq 0 ,~ \mu({\mathbb{R}})=M\right\}.\]
It is easily seen that $\rho^n(t)\in \mathcal{M}_M$ for all $t\geq 0$. Now, for a $\mu\in \mathcal{M}_M$ we consider its generalised inverse distribution function
\[X_\mu(z) \doteq \inf\left\{\vphantom{I^{\int}} x\in {\mathbb{R}} \colon \mu(]-\infty,x])> z\right\},\qquad z\in [0,M],\]
and define the \emph{rescaled $1$-Wasserstein distance} between $\mu_1,\mu_2\in \mathcal{M}_M$ as
\[d_{1}(\mu_1,\mu_2)\doteq\|X_{\mu_1}-X_{\mu_2}\|_{\mathbf{L^1}([0,M])}.\]

\begin{lemma}[Compactness of $\rho^n$ and $V^n$]\label{lem:check}
The sequences $(\rho^n)_{n\in{\mathbb{N}}}$ and $(V^n)_{n\in{\mathbb{N}}}$ are relatively compact in $\mathbf{L_{loc}^1}({\mathbb{R}}_+\times {\mathbb{R}})$.
\end{lemma}

\proof
We have a uniform (w.r.t.\ $n$ and $t$) bound for the total variation of $V^n(t)$ as obtained in Proposition~\ref{prop:v}. The uniform bound for the total variation of $W^n(t)$ obtained in the proof of Proposition~\ref{prop:w} then implies that $p(\rho^n(t))$ has uniformly bounded total variation, and the assumptions on $p$ listed in \eqref{eq:p} then imply that ${\mathrm{TV}}[\rho^n(t)]$ is uniformly bounded. Now, one can prove that there exists a constant $C>0$ such that
\[d_1(\rho^n(t),\rho^n(s))\leq C \, \left|t-s\right|.\]
The above estimate can be proven in the same way as in the proof of \cite[Proposition~8]{MarcoMaxARMA2015}, with few unessential changes that are left to the reader. Consequently, \cite[Theorem 5]{MarcoMaxARMA2015} implies the assertion for $\rho^n$. The assertion for $V^n$ then follows by the continuity of $p$ and by dominated convergence.
\endproof

\begin{theorem}[Convergence to weak solutions]\label{thm:main}
Let $(\bar{v},\bar{w})\in {\mathbf{BV}}({\mathbb{R}}\,;\,\mathcal{W})$ be such that $\bar\rho \doteq p^{-1}(\bar{w}-\bar{v})$ is compactly supported and belongs to $\mathbf{L^1}({\mathbb{R}};{\mathbb{R}}_+)$. Fix $n\in {\mathbb{N}}$ sufficiently large and let $N_n \doteq 2^n$, $\kappa_n \doteq 2^{-n}M$, with $M \doteq \left\|\bar{\rho}\right\|_{\mathbf{L^1}({\mathbb{R}})}$. Let $\bar{x}_0^n<\ldots<\bar{x}_{N_n}^n$ be the atomization constructed in \eqref{eq:initial_ftl}. Let $x_0^n(t),\ldots,x_{N_n}^n(t)$ be the solution to the multi population follow-the-leader system \eqref{eq:ftl} with initial datum $\bar{x}_0^n,\ldots,\bar{x}_{N_n}^n$. Let $\bar{w}_0^n, \ldots, \bar{w}_{N_n-1}^n$ be given by \eqref{eq:dw}. Set $W^n$ and $V^n$ as in \eqref{eq:hwn} and \eqref{eq:hvn} respectively, where $v^n_i$ and $y^n_i$ are defined by \eqref{eq:vni} and \eqref{eq:Deltan} respectively. Then, up to a subsequence, $(V^n,W^n)_{n\in{\mathbb{N}}}$ converges in $\mathbf{L_{loc}^1}({\mathbb{R}}_+\times {\mathbb{R}};\mathcal{W})$ as $n\rightarrow +\infty$ to a weak solution of the Aw, Rascle and Zhang system \eqref{eq:CauchyARZ} with initial datum $(\bar{v},\bar{w})$.
\end{theorem}

\proof Due to the result in Lemma~\ref{lem:check}, the sequence
\[\rho^n\doteq p^{-1}(W^n-V^n)\]
converges (up to a subsequence) a.e.\ and in $\mathbf{L_{loc}^1}({\mathbb{R}}_+\times {\mathbb{R}})$ to $\rho \doteq p^{-1}(w-v) \in \mathbf{L^1}\left({\mathbb{R}}_+ \times {\mathbb{R}}\right)$.

Now, let $\phi\in \mathbf{C_c^\infty}(]0,+\infty[\times {\mathbb{R}})$. We compute
\begin{align}
  & \int_{{\mathbb{R}}_+}\int_{\mathbb{R}} \rho^n(t,x)\left[\phi_t(t,x) + V^n(t,x) \phi_x(t,x)\right] \begin{pmatrix} 1 \\ W^n(t,x)\end{pmatrix} {\rm{d}} x \,{\rm{d}} t \label{eq:weak:discrete}\\
  =& \sum_{i=0}^{N_n-1}\int_{{\mathbb{R}}_+}y^n_i(t)\int_{x^n_i(t)}^{x^n_{i+1}(t)}
  \left[\phi_t(t,x) +v^n_i(y^n_i(t)) \phi_x(t,x)\right] \begin{pmatrix} 1 \\ \bar{w}^n_i\end{pmatrix} \,{\rm{d}} x \,{\rm{d}} t \nonumber\\
  =& \sum_{i=0}^{N_n-1}\int_{{\mathbb{R}}_+}y^n_i(t)\left[\frac{{\rm{d}}~}{{\rm{d}} t}\left(\int_{x^n_i(t)}^{x^n_{i+1}(t)}\phi(t,x) \,{\rm{d}} x\right)-v^n_{i+1}(y^n_{i+1}(t)) \, \phi(t,x^n_{i+1}(t))\right.\nonumber\\
  & \left.\vphantom{\int_{x^n_i(t)}^{x^n_{i+1}(t)}}+v^n_{i}(y^n_i(t)) \, \phi(t,x^n_{i}(t)) + v^n_i(y^n_i(t))\left[\phi(t,x^n_{i+1}(t))-\phi(t,x^n_{i}(t))\right]\right]\begin{pmatrix} 1 \\ \bar{w}^n_i\end{pmatrix} \,{\rm{d}} x \,{\rm{d}} t \nonumber\\
  =& \sum_{i=0}^{N_n-1}\int_{{\mathbb{R}}_+} \dfrac{y^n_i(t)^2}{\kappa_n} \left[v_{i+1}^n(y^n_{i+1}(t))-v_i^n(y^n_i(t))\right]
  \left[\int_{x^n_i(t)}^{x^n_{i+1}(t)}\phi(t,x) \,{\rm{d}} x\right]\begin{pmatrix} 1 \nonumber\\ \bar{w}^n_i\end{pmatrix} \,{\rm{d}} x \,{\rm{d}} t\\
  & -\sum_{i=0}^{N_n-1}\int_{{\mathbb{R}}_+}y^n_i(t)
  \left[v^n_{i+1}(y^n_{i+1}(t)) -  v^n_i(y^n_i(t))\right] \phi(t,x^n_{i+1}(t)) \begin{pmatrix} 1 \\ \bar{w}^n_i\end{pmatrix} \,{\rm{d}} x \,{\rm{d}} t \nonumber\\
  =& \sum_{i=0}^{N_n-1}\int_{{\mathbb{R}}_+}\frac{y^n_i(t)^2}{\kappa_n}
  \left[v_{i+1}^n(y^n_{i+1}(t))-v_{i}^n(y^n_{i}(t))\right]
  \left[\int_{x^n_i(t)}^{x^n_{i+1}(t)}
  \left[\phi(t,x)-\phi(t,x^n_{i+1}(t))\right] {\rm{d}} x\right]
  \begin{pmatrix} 1 \\ \bar{w}^n_i\end{pmatrix} \,{\rm{d}} t.\nonumber
\end{align}
Therefore, by observing that
\[
\left|\int_{x^n_i(t)}^{x^n_{i+1}(t)}
\left[\phi(t,x)-\phi(t,x^n_{i+1}(t))\right] {\rm{d}} x\right|
\le
\frac{{\mathrm{Lip}}[\phi]}{2} \left[x^n_i(t) - x^n_{i+1}(t)\right]^2
= \frac{{\mathrm{Lip}}[\phi]}{2} \frac{\kappa_n^2}{y^n_i(t)^2},
\]
and by recalling the uniform bound ${\mathrm{TV}}\left[V^n(t)\right] \leq C_v$ obtained in the proof of Proposition \ref{prop:v}, we easily get the estimate
\[
\left\| \int_{{\mathbb{R}}_+}\int_{\mathbb{R}} \rho^n(t,x)\left[\phi_t(t,x) + V^n(t,x) \phi_x(t,x)\right] \begin{pmatrix} 1 \\ W^n(t,x)\end{pmatrix} {\rm{d}} x \,{\rm{d}} t\right\|
  \leq  \kappa_n \, \frac{{\mathrm{Lip}}[\phi]}{2} \, C_v \, T \left[1+\left\|\bar{w}\right\|_{\mathbf{L^\infty}({\mathbb{R}})}\right],
  \]
where $T>0$ is such that ${\mathrm{supp}}(\phi)\subset [0,T]\times {\mathbb{R}}$.
Clearly, the right hand side above tends to zero as $n\rightarrow +\infty$. Finally, up to a subsequence, the double integral in \eqref{eq:weak:discrete} converges to
\[\int_{{\mathbb{R}}_+}\int_{\mathbb{R}}\rho(t,x) \left[\phi_t(t,x) + v(t,x) \phi_x(t,x)\right] \begin{pmatrix} 1 \\ w(t,x)\end{pmatrix} {\rm{d}} x \,{\rm{d}} t\]
and this concludes the proof.
\endproof

\section{Numerical experiments}\label{sec:numerics}
In order to test the proposed atomization algorithm, we consider four Riemann problems leading to four solutions of interest.
The first three coincide with those done in \cite[Section~4]{chalonsgoatin2007} and are used to check the ability of the scheme to deal with contact discontinuities.
The last one is the example given in \cite[Section~5]{AKMR2002} and is used to check the ability of the scheme to deal with the vacuum.
In each case, the method is first evaluated by means of a qualitative comparison of the approximate solution $\rho_\Delta$, $v_\Delta$, $w_\Delta$ with the exact solution $\rho$, $v$, $w$. Two initial mesh sizes are automatically determined by the total number of particles $N$ and the two density values of the Riemann data. The qualitative results corresponding to $N=500$ and final time $t=0.2$ for the Test 1, Test 2 and Test 3 and $t=1$ for Test 4 are presented on \figurename~\ref{fig:Test 1,2} and \figurename~\ref{fig:Test 3,4}.

\begin{figure}

\begin{minipage}[c]{.48\textwidth}
\begin{center}
\begin{tabular}{c}
    Test 1\\
  \hline
    $p(\rho)=1.4427\,\log(\rho)$,\\
    $\rho_\ell=0.9,~v_\ell=1$,\\
    $\rho_r=0.1,~v_r=1$,\\
  \hline
  \end{tabular}
\end{center}
\end{minipage}
\hfill
\begin{minipage}[c]{.48\textwidth}
\begin{center}
\begin{tabular}{c}
    Test 2\\
  \hline
    $p(\rho)=1.4427\,\log(\rho)$,\\
    $\rho_\ell=0.1,~v_\ell=1.8$,\\
    $\rho_r=0.2,~v_r=1.6$,\\
  \hline
  \end{tabular}
\end{center}
\end{minipage}
\bigskip\\
\begin{minipage}[c]{.48\textwidth}
\begin{center}
\includegraphics[width=\textwidth]{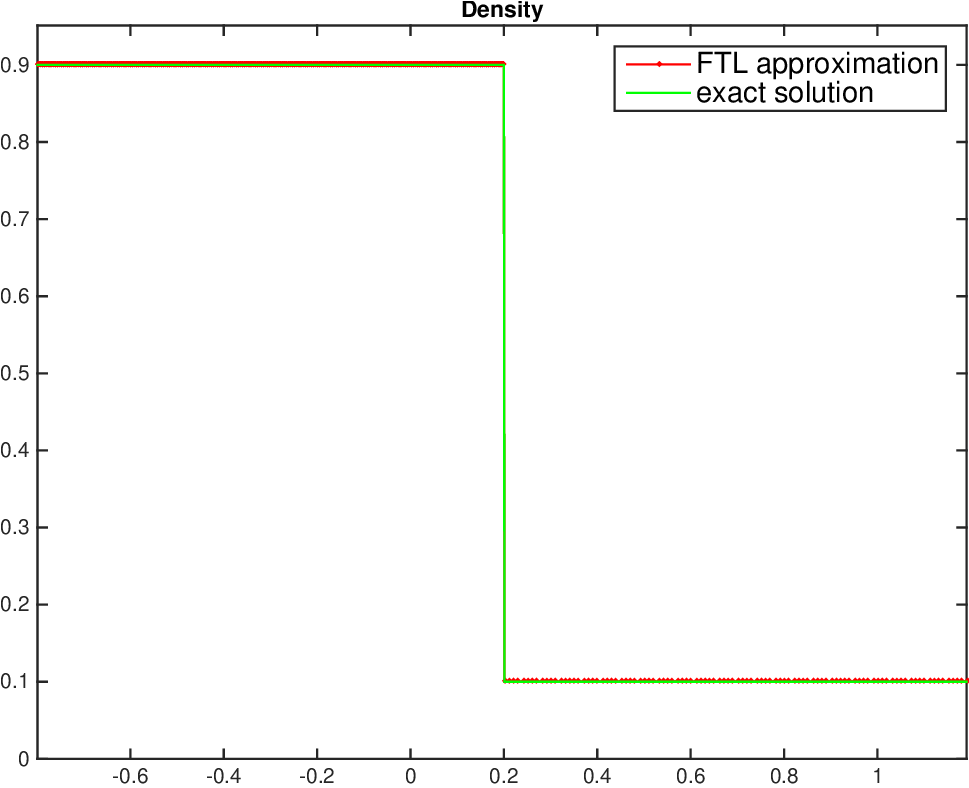}
\end{center}
\end{minipage}
\hfill
\begin{minipage}[c]{.48\textwidth}
\begin{center}
\includegraphics[width=\textwidth]{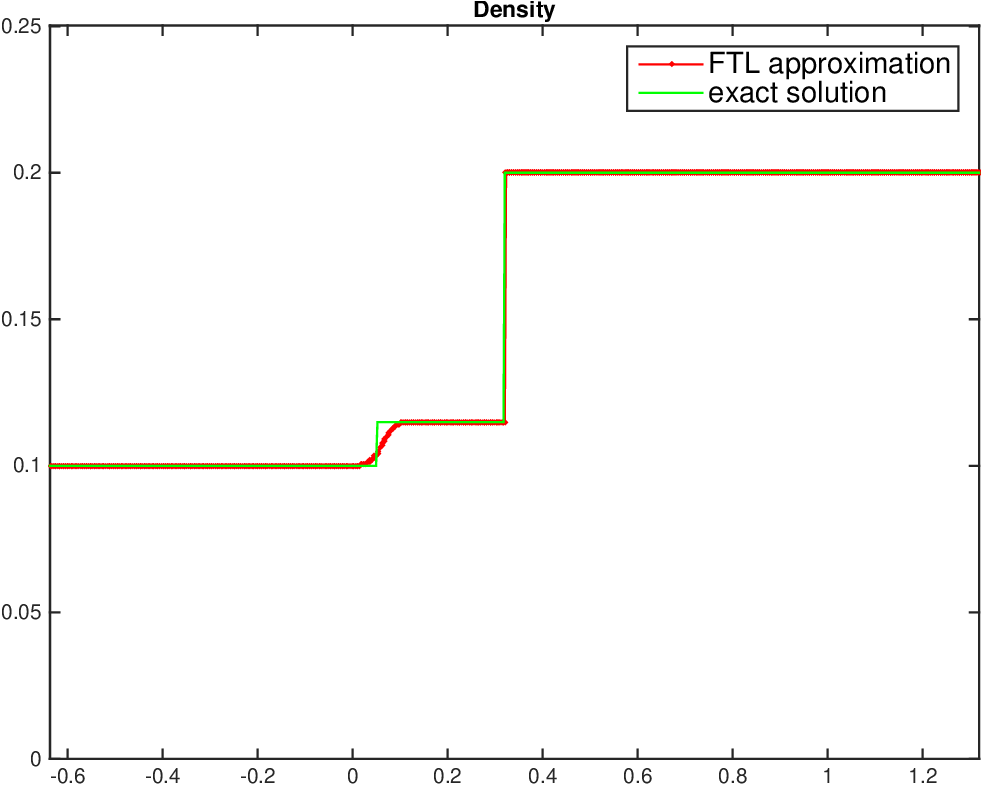}
\end{center}
\end{minipage}
\bigskip\\
\begin{minipage}[c]{.48\textwidth}
\begin{center}
\includegraphics[width=\textwidth]{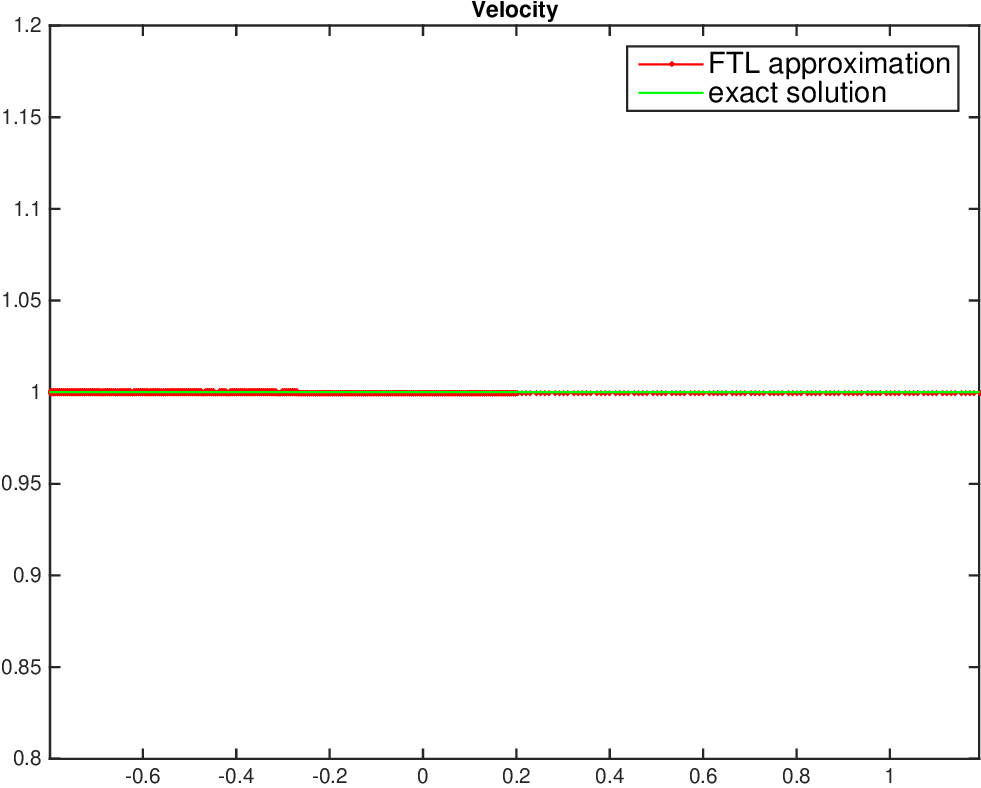}
\end{center}
\end{minipage}
\hfill
\begin{minipage}[c]{.48\textwidth}
\begin{center}
\includegraphics[width=\textwidth]{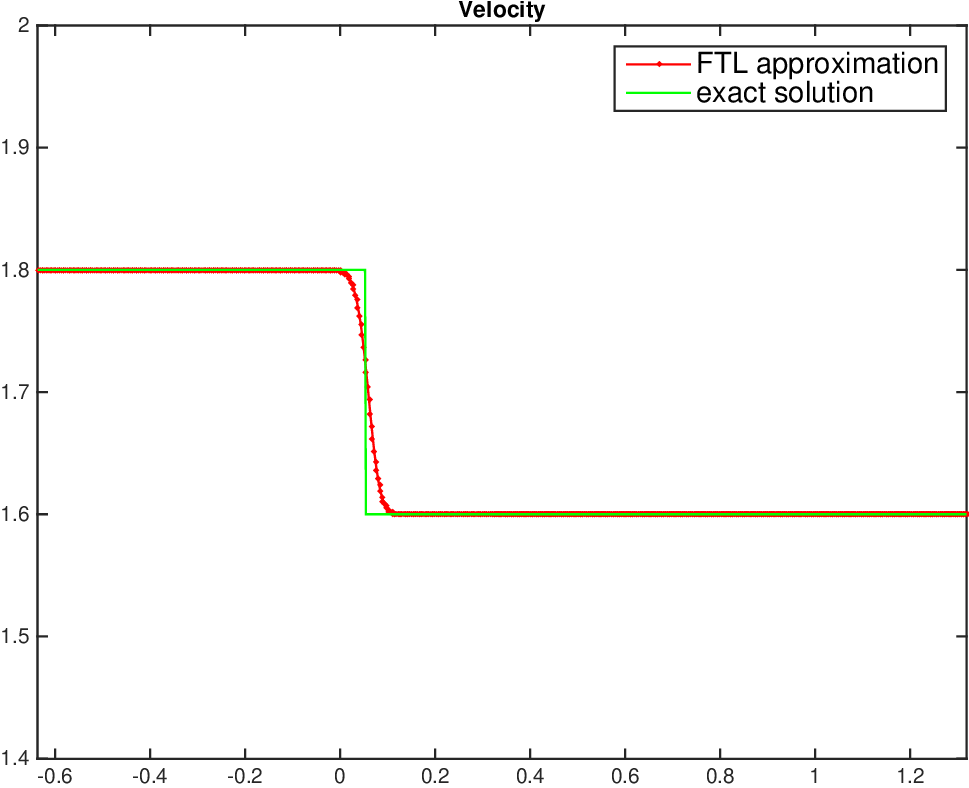}
\end{center}
\end{minipage}
\bigskip\\
\begin{minipage}[c]{.48\textwidth}
\begin{center}
\includegraphics[width=\textwidth]{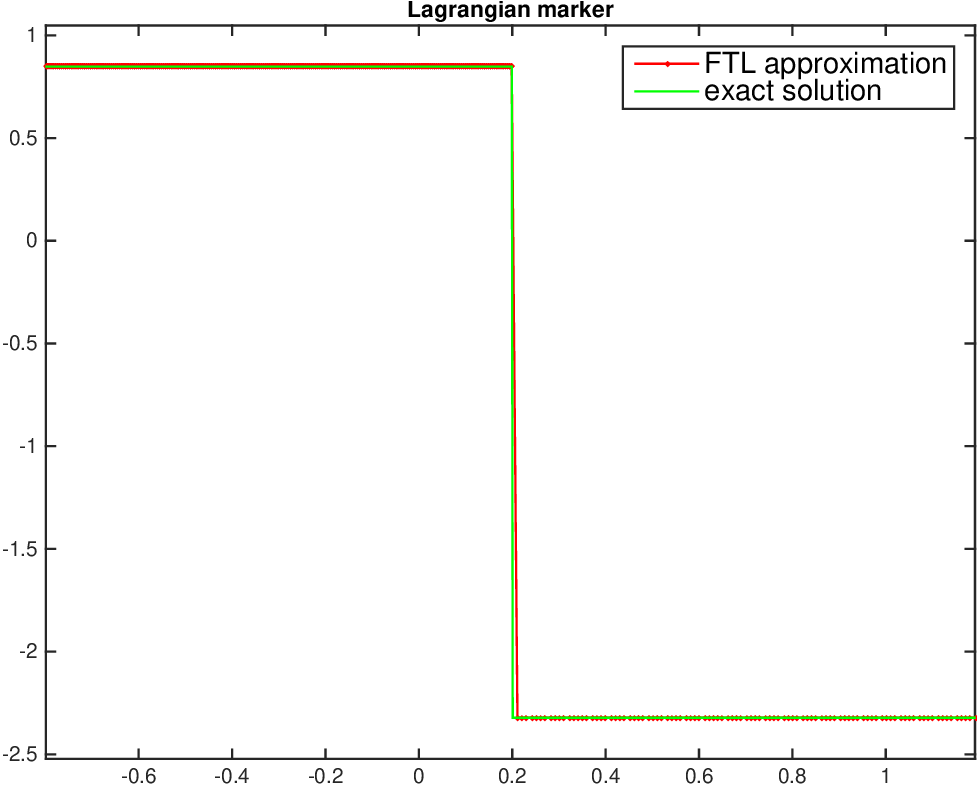}
\end{center}
\end{minipage}
\hfill
\begin{minipage}[c]{.48\textwidth}
\begin{center}
\includegraphics[width=\textwidth]{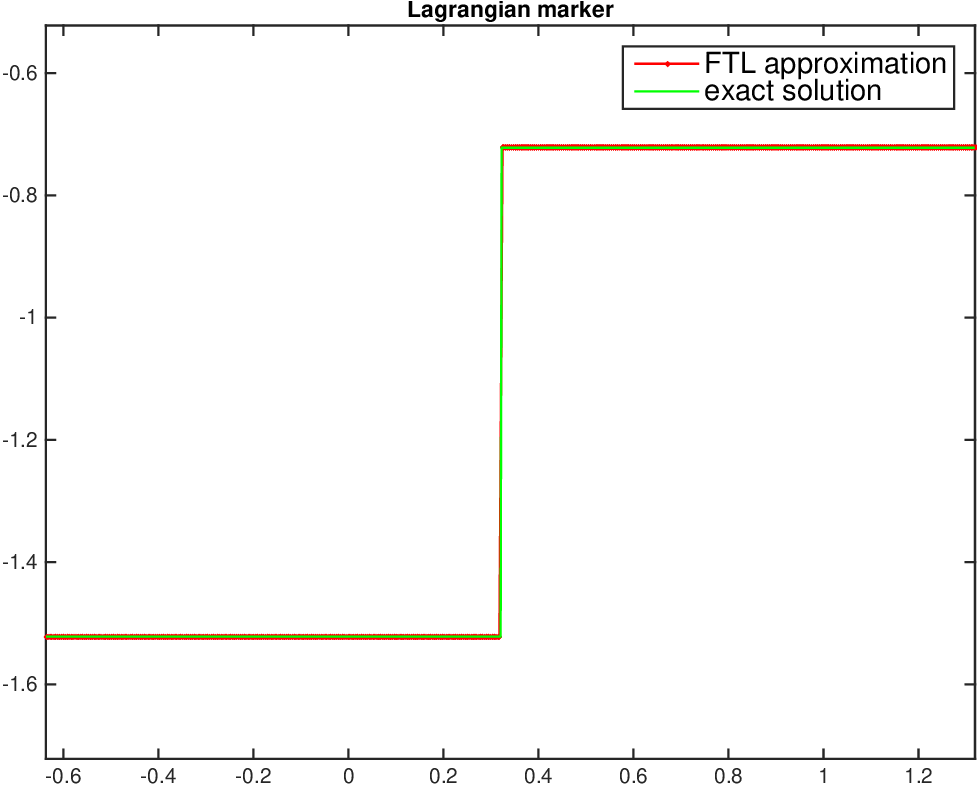}
\end{center}
\end{minipage}
\caption{Left column for Test 1 and right column for Test 2. Initial conditions are specified in the tables on the top using $N=200$ particles. \label{fig:Test 1,2}}
\end{figure}

\begin{figure}

\begin{minipage}[c]{.48\textwidth}
\begin{center}
\begin{tabular}{c}
    Test 3\\
  \hline
    $p(\rho)=1.4427\,\log(\rho)$,\\
    $\rho_\ell=0.5,~v_\ell=1.2$,\\
    $\rho_r=0.1,~v_r=1.6$,\\
  \hline
  \end{tabular}
\end{center}
\end{minipage}
\hfill
\begin{minipage}[c]{.48\textwidth}
\begin{center}
\begin{tabular}{c}
    Test 4\\
  \hline
    $p(\rho)=6\rho$,\\
    $\rho_\ell=0.05,~v_\ell=0.05$,\\
    $\rho_r=0.05,~v_r=0.5$,\\
  \hline
  \end{tabular}
\end{center}
\end{minipage}
\bigskip\\
\begin{minipage}[c]{.48\textwidth}
\begin{center}
\includegraphics[width=\textwidth]{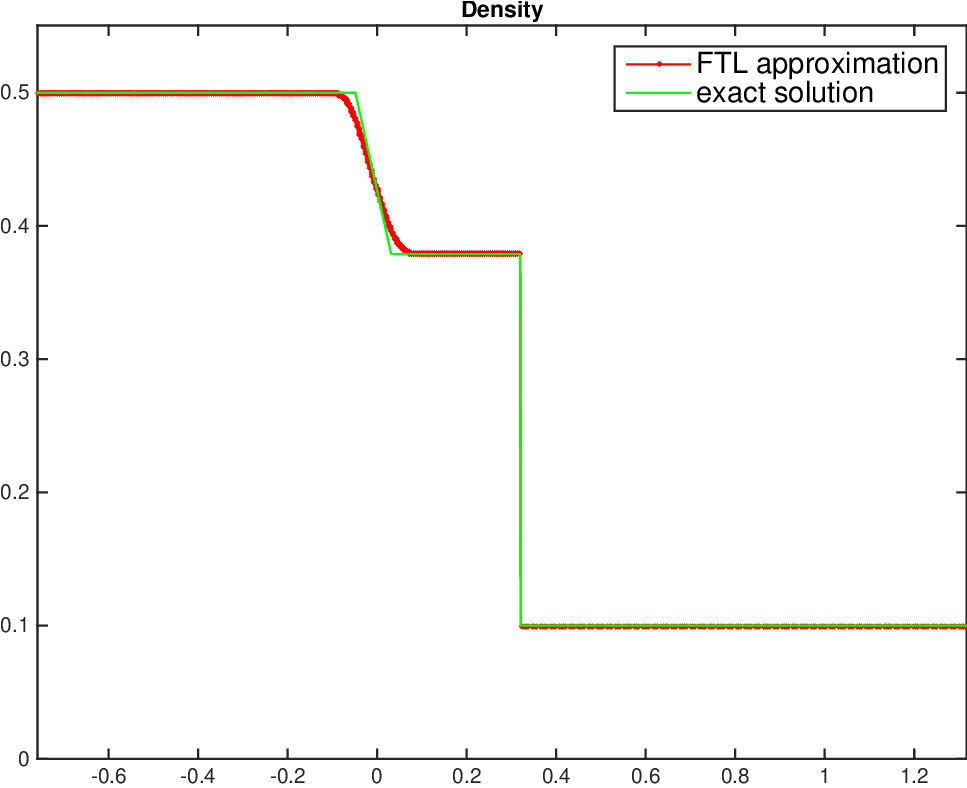}
\end{center}
\end{minipage}
\hfill
\begin{minipage}[c]{.48\textwidth}
\begin{center}
\includegraphics[width=\textwidth]{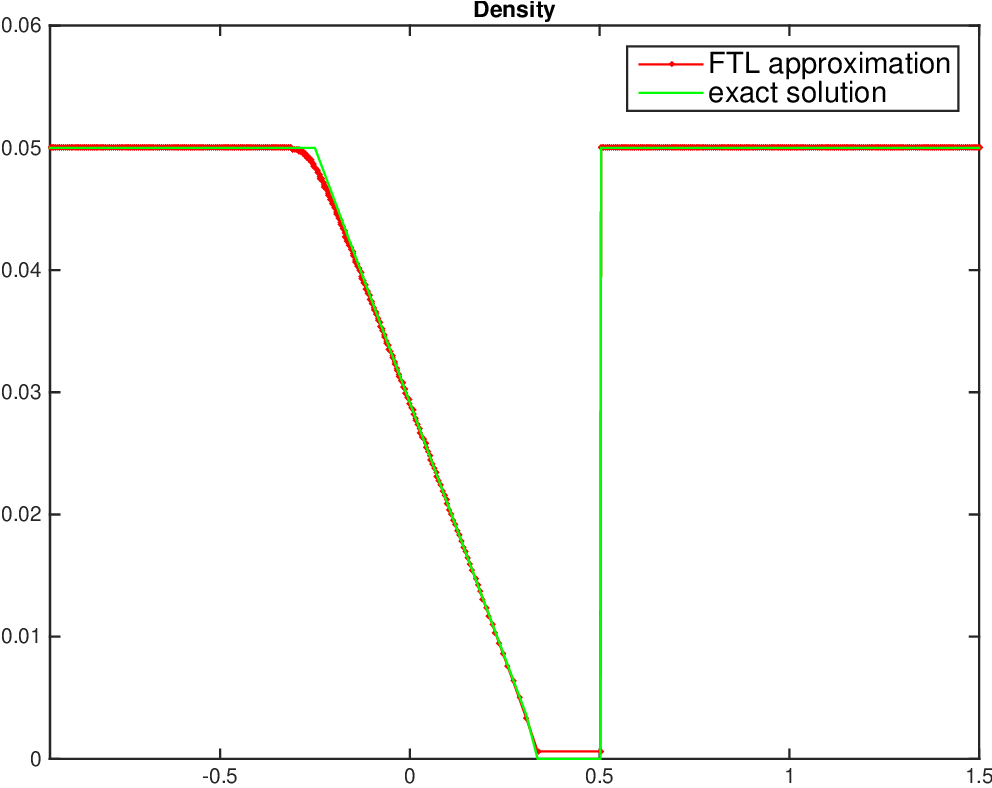}
\end{center}
\end{minipage}
\bigskip\\
\begin{minipage}[c]{.48\textwidth}
\begin{center}
\includegraphics[width=\textwidth]{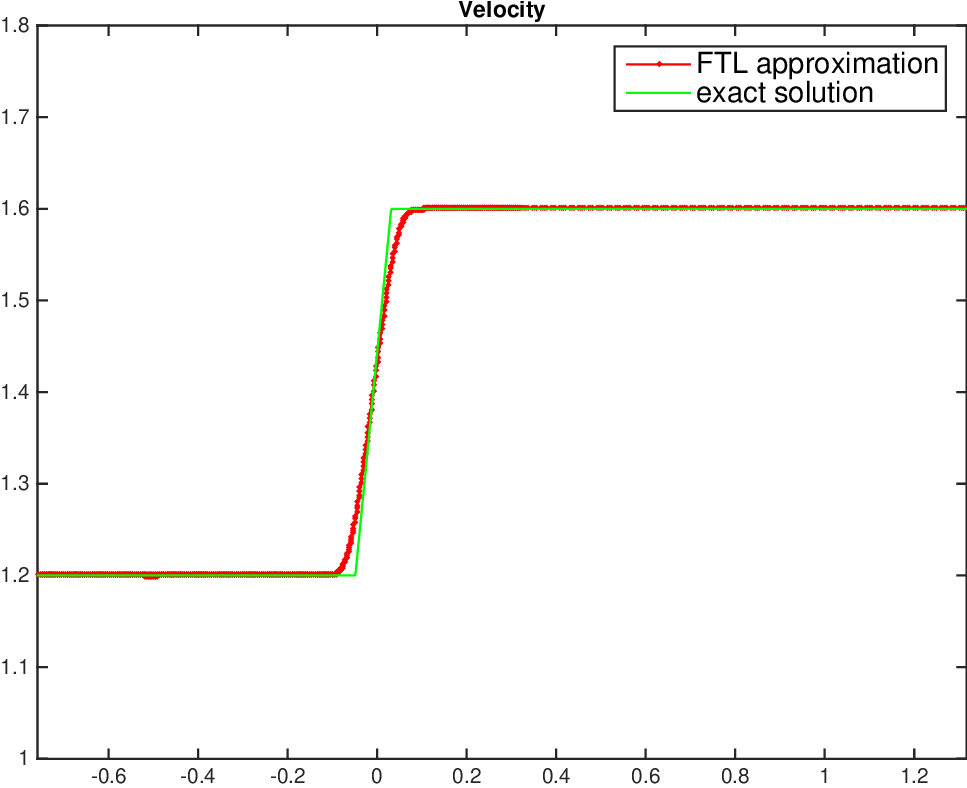}
\end{center}
\end{minipage}
\hfill
\begin{minipage}[c]{.48\textwidth}
\begin{center}
\includegraphics[width=\textwidth]{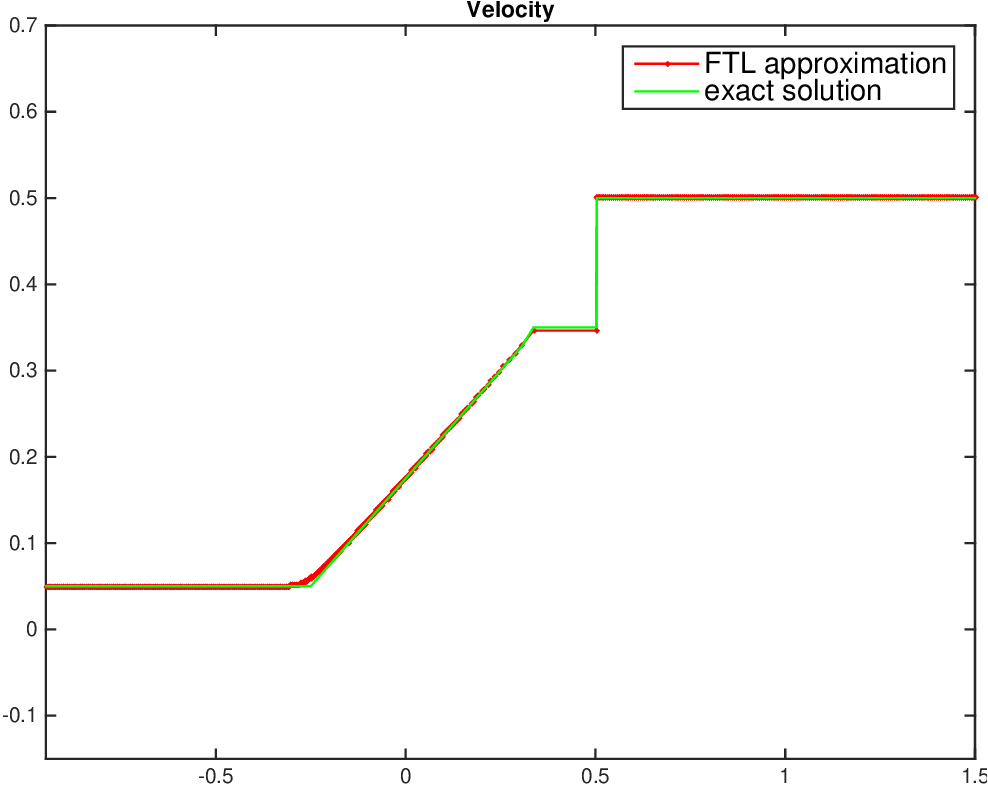}
\end{center}
\end{minipage}
\bigskip\\
\begin{minipage}[c]{.48\textwidth}
\begin{center}
\includegraphics[width=\textwidth]{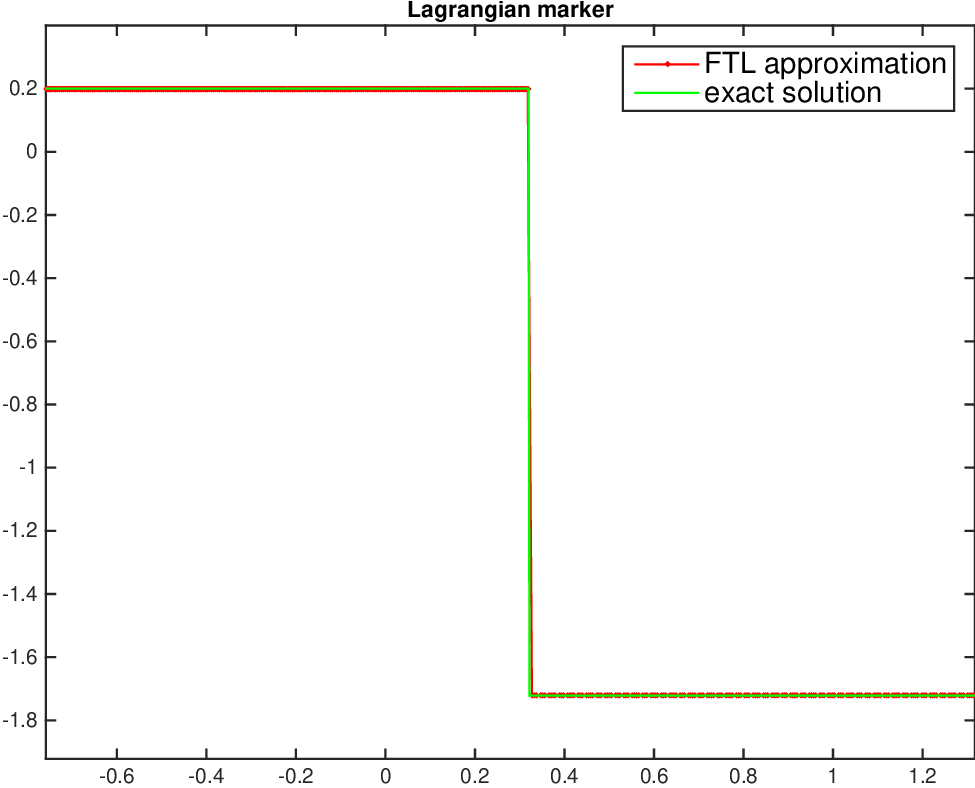}
\end{center}
\end{minipage}
\hfill
\begin{minipage}[c]{.48\textwidth}
\begin{center}
\includegraphics[width=\textwidth]{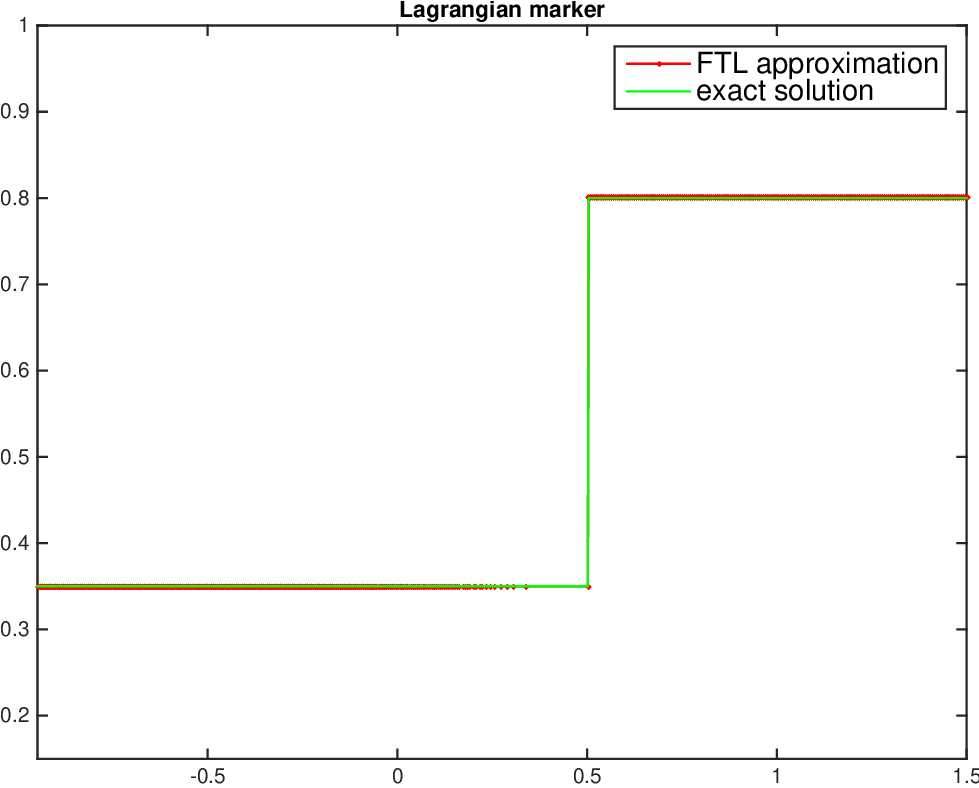}
\end{center}
\end{minipage}
\caption{Left column for Test 3 and right column for Test 4. Initial conditions are specified in the tables on the top using $N=200$ particles. \label{fig:Test 3,4}}
\end{figure}

Then, for several values of $N$, a quantitative evaluation through the $\mathbf{L^1}$-norm (of the difference between the exact and approximate densities solutions) is made, they are given on Table~\ref{errors}.
\begin{table}
\begin{center}
\begin{tabular}{c|c|c|c|c}
\hline
 $N$ & Test 1 & Test 2 & Test 3 & Test 4\\
 \hline
 \hline
 100 & $8.9e-03$ & $4.1e-03$ & $4.7e-03$ & $2.1e-03$\\
 \hline
 500 & $1.8e-03$ & $1.1e-03$ & $1.8e-03$ & $4.7e-04$\\
 \hline
 1000 & $4.7e-04$ & $5.7e-04$ & $1.2e-03$ & $2.5e-04$\\
 \hline
 2000 & $4.5e-04$ & $3.4e-04$ & $8.2e-04$ & $1.3e-04$\\
 \hline
\end{tabular}
 \caption{Different numbers of particles and corresponding discrete $\mathbf{L^1}$-errors for densities.\label{errors}}
\end{center}
\end{table}


\medskip
Received xxxx 20xx; revised xxxx 20xx.
\medskip

\end{document}